\documentclass[runningheads]{llncs}
\usepackage[T1]{fontenc}
\usepackage{graphicx}
\usepackage{amsfonts}
\usepackage{amsmath}
\usepackage{thm-restate}
\usepackage{xcolor}
\definecolor{Blue}{rgb}{0.2, 0.2, 0.6}

\usepackage{hyperref}
\hypersetup{
    colorlinks=true,
    linkcolor=Blue,
    filecolor=Blue,      
    urlcolor=Blue,
    citecolor =Blue
}

\begin{document}
\title{The Frobenius Problem for the Proth Numbers}
%
%
\author{Pranjal Srivastava\inst{1} \and Dhara Thakkar\inst{2}}
\authorrunning{\textbf{P. Srivastava and D. Thakkar}}
%
\institute{Indian Institute of Science Education and Research Bhopal \\ \email{pranjal.srivastava194@gmail.com} \\ \and Indian Institute of Technolgy Gandhinagar\\ \email{thakkar\_dhara@iitgn.ac.in}}
\maketitle              
\begin{abstract}
Let $n$ be a positive integer greater than $2$. We define \textit{the Proth numerical semigroup}, $P_{k}(n)$, generated by $\{k 2^{n+i}+1 \,\mid\, i \in \mathbb{N}\}$, where $k$ is an odd positive number and $k < 2^{n}$. In this paper, we introduce the Frobenius problem for the Proth numerical semigroup $P_{k}(n)$ and give formulas for the embedding dimension of $P_{k}(n)$. We solve the Frobenius problem for $P_{k}(n)$ by giving a closed formula for the Frobenius number. Moreover, we show that $P_{k}(n)$ has an interesting property such as being Wilf.

\end{abstract}

\section{Introduction}

The mathematician Ferdinand Frobenius defines the problem that asks to find the largest integer that is not expressible as a non-negative integer linear combination of elements of $L$, where $L$ is a set of $m$ coprime positive integers.

The Frobenius problem is defined as follows: Given a set $L=\{l_1,l_2,...,l_m\}$ of coprime positive integers and $l_i \geq 2$, find the largest natural number that is not expressible as a non-negative linear combination of $l_1,l_2,...,l_m$. It is also known as the money exchange or coin exchange problem in number theory. In literature, the connection between graph theory, theory of computer science and Frobenius problem has been developed (see \cite{heap-linear,hujter-exact,raczunas-diophantine,owens-algorithm}). This is because the Frobenius problem has attracted mathematicians as well as computer scientists since the $19$-th century
(see  \cite{alfonsin-diophantine}, Chapter 1 in \cite{beck-computing}, Problem C7 in \cite{guy-unsolved}, \cite{sylvester-subvariants}). 

For the special case e.g., $m=2$, the explicit formula to find the Frobenius number is known, it is $l_1 l_2-l_1-l_2$ proved in \cite{sylvester-mathematical}. In addition to that, for the case $m=3$, semi-explicit formula is known to find the Frobenius number \cite{robles-frobenius}. Moreover, R\"{o}dseth \cite{rodseth-linear}, Selmer \cite{selmer-linear} and Beyer \cite{beyer-linear} have developed algorithms to solve the Frobenius problem in the case $m = 3$. In 1996, Ram\'{i}rez-Alfons\'{i}n showed that the Frobenius problem for variable $m$ is NP-hard \cite{ramirez-complexity}.

The Frobenius problem has been studied for several special cases, e.g., numbers in a geometric sequence, arithmetic sequence, Pythagorean triples, three consecutive squares or cubes \cite{tripathi-frobenius,tripathi-arithmetic,gil-frobenius,lepilov-frobenius}. Moreover, the Frobenius problem is defined on some special structure like Numerical semigroup (see the definition below). 

Let $\mathbb{N}$ and $\mathbb{Z}$ be the set of non-negative integers and set of integers, respectively. A subset $S$ of $\mathbb{N}$ containing $0$ is a \emph{numerical semigroup} if $S$ is closed under addition and has a finite complement in $\mathbb{N}$. If $S$ is a numerical semigroup and $S=\langle B \rangle$, then we call $B$, a system of generators of $S$. A system of generators $B$ of $S$ is minimal if no proper subset of $B$ generates $S$. In \cite{rosales-numerical} Rosales et al. proved that every numerical semigroup admits a unique minimal system of generators and such a system is finite. The cardinality of a minimal system of generators of $S$ is called the \emph{embedding dimension} of $S$ denoted by $\mathrm{e}(S)$.

The Frobenius number of a numerical semigroup $S=\langle \{a_{1},a_{2},\dots,a_{n}\} \rangle$ (denoted by $\mathrm{F}(S)$) is the greatest integer that cannot be expressed as a sum $\sum\limits_{i=1}^{n}t_{i}a_{i}$, where $t_{1},\dots,t_{n} \in \mathbb{N}$ \cite{rosales-numerical,assi-numerical}.


To solve the Frobenius problem for numerical semigroups, several methods were introduced, e.g., see \cite{bras-bounds,rosales-unique,rosales-numerical,rosales-Fundamental}. In particular, in recent articles, the method of computing the Ap\'{e}ry set (see  Definition \ref{definition-Apery}) and deduce the Frobenius number using the Ap\'{e}ry set has been presented. In literature, there exists a large list of publications devoted to solve the Frobenius problem for special classes of numerical semigroup, including the Frobenius problem for Fibonacci numerical semigroup \cite{marin-frobenius}, Mersenne numerical semigroup \cite{Mersenne}, Thabit numerical semigroup \cite{Thabit} and repunit numerical semigroup \cite{Repunit}. We note that the study of the Frobenius number for the mentioned numerical semigroups has been inspired by special primes such as Fibonacci, Mersenne, Thabit and repunit primes. In this paper, we introduce Proth numerical semigroup motivated by the Proth number. The main aim of this paper is to study the Proth numerical semigroup and its invariants like embedding dimension, Frobenius number, etc. 

  
In number theory, the \emph{Proth number} (named in honor of the mathematician Fran\c{c}ois Proth) is a natural number of the form $k2^{n}+1$, where $n$ and $k$ are positive numbers and $k < 2^{n}$ is an odd number. We say that a Proth number is a \emph{Proth prime} if it is prime. 
  
A numerical semigroup $S$ is the \emph{Proth numerical semigroup} if $n \in \mathbb{N}$ such that $S=\langle \{k 2^{n+i}+1 \,\mid\, i \in \mathbb{N}\}\rangle$, where $n$ and $k$ are positive numbers and $k < 2^{n}$ is an odd number. We denote by $P_{k}(n)$ the numerical semigroup $\langle \{k 2^{n+i}+1 \,\mid\, i \in \mathbb{N}\}\rangle$. It is easy to see that when $k=1$ the Proth numerical semigroup is the Cunningham numerical semigroup \cite{song-frobenius}. Hence, we can assume that $2^r < k < 2^{r+1}$ for some $r$.

In this paper, we first prove that $\mathrm{e}(P_{k}(n))$ is $n+r+1$ where $2^r < k < 2^{r+1}$. Later, we find the Frobenius number of the Proth numerical semigroup. More formally, we prove the following theorem. 

\begin{theorem}
Let $n>2$ be a positive integer. Then $\mathrm{F}(P_{2^r+1}(n))=2s_{1}+s_{n}+s_{n+r}-s_{0}$, where $s_i=k 2^{n+i}+1$ for $i \in \mathbb{N}$.
\end{theorem}

Let $S$ be a numerical semigroup. An integer $x$ is a \emph{pseudo-Frobenius number} of $S$ if $x \in \mathbb{Z}\setminus S$ and $x+s \in S$ for all $s \in S \setminus \{0\}$. The set of pseudo-Frobenius numbers of $S$ is denoted by $\mathrm{PF}(S)$, and the cardinality of the set $\mathrm{PF}(S)$ is called the \emph{type} of $S$ denoted by $\mathrm{t}(S)$ \cite{rosales-numerical,assi-numerical}. 
  
We find the set of pseudo-Frobenius numbers of the Proth numerical semigroup $P_{2^{r}+1}(n)$ and prove that its type is $n+r-1$.

In the context of a numerical semigroup, it is reasonable to study the problems that connect the Frobenius number and other invariants of a numerical semigroup. One such problem posed by Wilf (known as Wilf's conjecture) in \cite{Wilf-conjecture} is as follows: Let $S$ be a numerical semigroup and $\nu(S)=|\{s \in S \,\mid\, s \leq \mathrm{F}(S)\}|$, is it true that $\mathrm{F}(S)+1 \leq e(S)\nu(S)$, where $e(S)$ is the embedding dimension and $\mathrm{F}(S)$ is the Frobenius number of $S$? Note that the numerical semigroups that satisfy Wilf's conjecture are called Wilf.

The conjecture is still open; in spite of it, an affirmative answer has been given for a few special classes of a numerical semigroup. In this paper,  we prove that the Proth numerical semigroup $P_{2^{r}+1}(n)$ supports Wilf's conjecture. 

This paper is an attempt to understand the Frobenius problem and Wilf conjecture for arbitrary embedding dimension through the Proth numerical semigroup. Our approach was inspired by the ideas discussed in \cite{Mersenne,Thabit}. However, it is worth noting that our techniques to find the Ap\'{e}ry set of the Proth numerical semigroups differ from the existing ones \cite{Mersenne,Thabit}.

The reader not familiarized with the study of numerical semigroup and the terminologies like embedding dimension, pseudo-Frobenius numbers, type, etc., can refer to the literature \cite{rosales-numerical,assi-numerical}. 

\section{The Embedding Dimension}
We begin this section by proving that $P_{k}(n)$ is a numerical semigroup. Later, we prove that the embedding dimension of $P_{k}(n)$ is $n+r+1$. Some of the techniques used in this section are introduced earlier see, e.g., \cite{Mersenne,Thabit,song-frobenius}. 

\begin{lemma}(Lemma 2.1 in \cite{rosales-numerical})\label{gcd-lemma}
Let $S$ be a nonempty subset of $\mathbb{N}$. Then $\langle S \rangle$ is a numerical semigroup if and only if $\mathrm{gcd}(S)=1$.
\end{lemma}

\begin{theorem}
Let $n > 2$ be an integer, then $P_{k}(n)$ is a numerical semigroup.
\end{theorem}
\begin{proof}
It is clear that $P_{k}(n)\subseteq \mathbb{N}$ is closed under addition and contains zero. Note that from Lemma \ref{gcd-lemma} it is enough to show that $\mathrm{gcd}(P_{k}(n))=1$. Let $k2^{n}+1$, $k2^{n+1}+1 \in P_{k}(n)$. Then $\mathrm{gcd}(k2^{n}+1, k2^{n+1}+1)=\mathrm{gcd}(k2^{n}+1, k2^{n+1}-k2^{n})=\mathrm{gcd}(k2^{n}+1, k2^{n})=1$. Therefore, $P_{k}(n)$ is a numerical semigroup.  \qed
\end{proof}

Next we give the minimal system of generators of the Proth numerical semigroup. To this purpose, we need some preliminary results.

\begin{lemma}\label{gen}(Lemma 2.1 in \cite{song-frobenius})
Let $S$ be a numerical semigroup generated by a non-empty set $M$ of positive integers. Then the following conditions are equivalent:
\begin{enumerate}
    \item[(i)] $2m-1 \in S$ for all $m \in M$;
    \item[(ii)] $2s-1 \in S$ for all $s \in S\setminus \{0\}$.
\end{enumerate}
\end{lemma}


\begin{theorem}\label{generator}
Let $n > 2$ be an integer, then $P_{k}(n)=\langle \{k2^{n+i}+1 \,\mid\, i=0,\dots,n+r\}\rangle$.
\end{theorem}
\begin{proof}
 Let $P=\langle \{ k2^{n+i}+1 \,\mid\, i \in \{0,1,\dots,n+r\}\big\} \rangle$. It is clear that $P\subseteq P_{k}(n)$. To prove the other direction it is enough to prove that $k 2^{n+i}+1 \in P$ for all $i \in \mathbb{N}$. Let $i\in \{0,1,\dots,n+r-1\}$, then $2(k 2^{n+i}+1)-1=k 2^{n+i+1}+1 \in P$. For $i=n+r$, $2(k2^{n+n+r}+1)-1=((k-2^{r})2^{n+1}+3)(k 2^{n}+1)+((2^{r+1}-k)2^{n}-2)(k 2^{n+1}+1)  \in P $. From Lemma \ref{gen}, we get $2s-1 \in P$ for all $s \in P \setminus \{0\}$. By induction, we can deduce that $k 2^{n+i}+1 \in P$ for all $i \geq n+r+1$ and hence $P_{k}(n)=\langle \{k 2^{n+i}+1 \,\mid\, i=0,\dots,n+r\}\rangle$. \qed  
\end{proof}

Note that, Theorem \ref{generator} tells us that $\{k 2^{n+i}+1 \,\mid\, i=0,\dots,n+r\}$ is a system of generators of $P_{k}(n)$.


\begin{lemma}\label{last element of GS does not belongs to}
Let $n > 2$ be an integer, then $k2^{n+n+r}+1 \notin \langle \{ k2^{n+i}+1 \,\mid\, i \in \{0,1,\dots,n+r-1\}\big\} \rangle$.  
\end{lemma}

\begin{proof}
Assume to the contrary that there exists $a_{0},a_{1},\dots,a_{n+r-1} \in \mathbb{N}$ such that 
\begin{align*}
    k2^{n+n+r}+1=&\sum\limits_{i=0}^{n+r-1}a_{i}(k2^{n+i}+1)\\
    =&k2^{n}\Big(\sum_{i=0}^{n+r-1}2^{i}a_{i}\Big)+\sum\limits_{i=0}^{n+r-1}a_{i}.
\end{align*}
Hence, $\sum_{i=0}^{n+r-1}a_{i}=1(\mathrm{mod}\,k2^{n})$ and we get, $\sum_{i=0}^{n+r-1}a_{i}=t k 2^{n}+1$ for some $t \in \mathbb{N}$. Observe that $t \neq 0$. Thus, $\sum_{i=0}^{n+r-1}a_{i} \geq k2^{n}+1$. Therefore, $k2^{n+n+r}+1=\sum\limits_{i=0}^{n+r-1}a_{i}(k2^{n+i}+1) \geq (\sum_{i=0}^{n+r-1}a_{i})(k2^{n}+1) \geq (k2^{n}+1)^{2}$. Since  $2^{r}< k $ we get
\begin{align*}
2^{r+n}<2^{n}k<2^{n}k+2 &\Rightarrow
    k2^{r+n+n} < k^{2}2^{2n}+2k2^{n}\\
    &\Rightarrow
    k2^{r+n+n}+1 < k^{2}2^{2n}+2k2^{n}+1\\
    &\Rightarrow 
    k2^{n+n+r}+1 < (k2^{n}+1)^{2}.
\end{align*}
Hence, $k2^{n+n+r}+1 \geq (k2^{n}+1)^{2}>k2^{n+n+r}+1$, which is a contradiction. Therefore, $k2^{n+n+r}+1 \notin \big\langle \big\{ k2^{n+i}+1 \,\mid\, i \in \{0,1,\dots,n+r-1\}\big\} \big\rangle$.  \qed
\end{proof}

\begin{theorem}
Let $n > 2$ be an integer and let $P_{k}(n)$ be the Proth numerical semigroup associated to $n$, then $\mathrm{e}(P_{k}(n))=n+r+1$. Moreover, $\{ k2^{n+i}+1 \,\mid\, i \in \{0,1,\dots,n+r\}\big\}$ is the minimal system of generators of $P_{k}(n)$.
\end{theorem}
\begin{proof}
By Theorem \ref{generator}, we know that $\{k 2^{n+i}+1 \,\mid\, i \in 0,1,\dots,n+r\}$ is a system of generator for $P_{k}(n)$. Suppose that it is not minimal system of generators of $P_{k}(n)$. Then there exists $l \in \{1,2,\dots,n+r-1\}$ such that $k 2^{n+l}+1 \in \langle k 2^{n+i}+1 \,\mid\, i \in \{0,1,\dots,l-1\}\rangle$. Let $T=\langle k 2^{n+i}+1 \,\mid\, i \in \{0,1,\dots,l-1\}\rangle$. If $i \in \{0,1,\dots,l-2\}$, then $2(k 2^{n+i}+1)-1=k 2^{n+i+1}+1 \in T$ and $2(k 2^{n+l-1}+1)-1=k 2^{n+l}+1 \in T$. From Lemma \ref{gen}, we have $2t-1\in T$ for all $t \in T\setminus \{0\}$. Hence, by induction we can obtain that $k 2^{n+i}+1\in T$ for all $i \geq l$, which is a contradiction as $k 2^{n+n+r}+1 \notin T$ from Lemma \ref{last element of GS does not belongs to}. Therefore, $\{ k2^{n+i}+1 \,\mid\, i \in \{0,1,\dots,n+r\}\big\}$ is the minimal system of generators of $P_{k}(n)$ and $\mathrm{e}(P_{k}(n))=n+r+1$.   \qed 
\end{proof}

\section{The Ap\'{e}ry Set}\label{section-apery-set}
In this section, we study the notion of Ap\'{e}ry set and give the explicit description of the elements of the Ap\'{e}ry set of the Proth numerical semigroup $P_{2^{r}+1}(n)$ for all $r \geq 1$. We denote by $s_i$ the element $k2^{n+i}+1$ for all $i \in \mathbb{N}$. Thus, with this notation, $\{s_0, s_1, ..., s_{n+r}\}$ is the minimal system of generators of $P_{k}(n)$.

\begin{definition}\cite{AperySet,rosales-numerical}\label{definition-Apery}
Let $S$ be a numerical semigroup and $n \in S\setminus \{0\}$. The Ap\'{e}ry set of $S$ with respect to $n$ is $\mathrm{Ap}(S,n)=\{ s \in S \,\mid\, s-n \notin S \}$.
\end{definition}

It is clear from the following lemma that $|\mathrm{Ap}(S,n)|=n$.

\begin{lemma}(Lemma 2.4 in \cite{rosales-numerical})\label{size-of-apery-set}
Let $S$ be a numerical semigroup and let $n$ be a nonzero element of $S$. Then $\mathrm{Ap}(S,n)=\{w(0),w(1),\dots,w(n-1)\}$, where $w(i)$ is the least element of $S$ congruent with $i$ modulo $n$, for all $i \in \{0,\dots,n-1\}$.
\end{lemma}

Our next goal is to describe the elements of $\mathrm{Ap}(P_{k}(n),s_0)$. 

\begin{restatable}{lemma}{RelationSiSj}\label{relation-Si-Sj}
Let $n > 2$ be an integer. Then:
\begin{enumerate}
    \item[(1)] if $ 0 < i \leq j <n+r$ then $s_{i}+2s_{j}=2s_{i-1}+s_{j+1}$;
    \item[(2)] if $0 < i \leq n+r $ then $s_{i}+2s_{n+r}=2s_{i-1}+\alpha s_{0}+\beta s_{1}$, where $\alpha=(k-2^{r})2^{n+1}+3$ and $\beta=(2^{r+1}-k)2^{n}-2$. 
\end{enumerate}
\end{restatable}
\begin{proof}
$(1)$ If $0 < i \leq j <n+r$ then we have
   \begin{align*}
        s_{i}+2s_{j}= & (k2^{n+i}+1)+2(k2^{n+j}+1)\\
        = & 2(k2^{n+i-1}+1)+(k2^{n+j+1}+1)
        = 2s_{i-1}+s_{j+1}.
    \end{align*}
    
$(2)$ If $0 < i \leq n+r $ then we get
    \begin{align*}
      s_{i}+2s_{n+r}=&(k2^{n+i}+1)+2(k2^{n+n+r}+1)\\
      =&2(k2^{n+i-1}+1)+k2^{2n+r+1}+1\\
      =&2s_{i-1}+\alpha(k2^{n}+1)+\beta(k2^{n+1}+1)=2s_{i-1}+\alpha s_0+\beta s_1,
    \end{align*}
    \quad where $\alpha=(k-2^{r})2^{n+1}+3, \beta=(2^{r+1}-k)2^{n}-2$.   \qed
\end{proof}

Let $P(r,n)$ denotes the set of all $n+r$-tuple $(a_{1},\dots,a_{n+r})$ that satisfies the following conditions:
\begin{enumerate}
    \item for every $i \in \{1,\dots,n+r\}$, $a_{i}\in \{0,1,2\}$;
    \item if $a_{j}=2$ for some $j=2,\dots,n+r$ then $a_{i}=0$ for $i<j$. 
\end{enumerate}

\begin{lemma}(Lemma 3.3 in  \cite{gu-frobenius})\label{card-P}
The cardinality of $P(r,n)$ is equal to $2^{n+r+1}-1$.
\end{lemma}
   
\begin{lemma}\label{AP}
Let $n > 2$ be an integer and let $P_{2^{r}+1}(n)$ be the Proth numerical semigroup minimally generated by $\{s_{0},s_{1},\dots,s_{n+r}\}$. If $s \in \mathrm{Ap}(P_{2^{r}+1}(n),s_{0})$ then there exist $(a_{1},\dots,a_{n+r})\in P(r,n)$ such that $s=a_{1}s_{1}+\dots+a_{n+r}s_{n+r}$.
\end{lemma}
\begin{proof}
 Let $s \in \mathrm{Ap}(P_{2^{r}+1}(n),s_{0})$. We prove the result of lemma using induction on $s$.  When $s=0$ then result follows trivially. Assume that $s > 0$ and $j$ be the smallest element from $\{0,1,\dots,n +r\} $ such that $ s-s_{j} \in P_{2^{r}+1}(n)$. Since $s \in  \mathrm{Ap}(P_{2^{r}+1}(n),s_{0})$ we have $j \neq 0$ and $s-s_{j} \in \mathrm{Ap}(P_{2^{r}+1}(n),s_{0})$. Now from induction hypothesis there exist $(a_{1},\dots,a_{n+r}) \in P(r,n)$ such that $s-s_{j}=a_{1}s_{1}+a_{2}s_{2}+\dots+a_{n+r}s_{n+r}$, hence $s=a_{1}s_{1}+a_{2}s_{2}+\dots+(a_{j}+1)s_{j}+\dots+a_{n+r}s_{n+r}$. Note that, to conclude the proof it suffices to prove that $(a_{1},\dots,a_{j}+1,\dots,a_{n+r}) \in P(r,n)$. 

(1) To prove $(a_{1},a_{2},\dots,a_{j}+1,\dots,a_{n+r}) \in \{0,1,2\}^{n+r}$, it is enough to show that $a_{j}+1 \neq 3$. If $a_{j}+1=3$ then from Lemma \ref{relation-Si-Sj},
\begin{itemize}
    \item[(i)] for $j<n+r$, we have $s_{j}+2s_{j}=2s_{j-1}+s_{j+1}$. This implies that, \\
    $s-s_{j-1} =a_{1}s_{1}+ \dots+s_{j-1}+(a_{j+1}+1)s_{j+1}+\dots+a_{n+r}s_{n+r}$.
    
    \item [(ii)] for $j=n+r$, we have, $s_{j}+2s_{j}=2s_{j-1}+\alpha s_{0}+\beta s_{1}$. This implies that,\\
    $s-s_{j-1} =\alpha s_{0}+(a_{1}+\beta) s_{1}+a_{2}s_{2}+ \dots+(a_{n+r-1}+1)s_{n+r-1}$.
\end{itemize}
In both the cases, we get $s-s_{j-1} \in P_{2^{r}+1}$, which is a contradiction to the minimality of $j$. Hence, $a_{j}+1\neq 3$.

(2) From the minimality of $j$, we obtain that $a_i=0$ for all $1 \leq i < j$. Now assume that there exist $l > j$ such that $a_{l}=2$, then again from Lemma \ref{relation-Si-Sj}, we have 
\begin{itemize}
    \item[(i)] for $l<n+r$, we have $s_{j}+2s_{l}=2s_{j-1}+s_{l+1}$;
    \item [(ii)] for $l=n+r$, we have, $s_{j}+2s_{l}=2s_{j-1}+\alpha s_{0}+\beta s_{1}$.
\end{itemize}
Again by the same argument as in (1), we have $s-s_{j-1} \in P_{2^{r}+1}$, which contradict the minimality of $j$. 

Therefore, $(a_{1},\dots,a_{j}+1,\dots,a_{n+r}) \in P(r,n)$. \qed
\end{proof}


It follows from Lemma \ref{AP} that  $\mathrm{Ap}(P_{2^{r}+1}(n),s_{0}) \subseteq \{a_{1}s_{1}+\dots+a_{n+r}s_{n+r} \,\mid\, (a_{1},\dots,a_{n+r})\in P(r,n)\}$. 

The next remark tells that the equality in the above expression does not hold in general.  
\allowdisplaybreaks
\begin{remark}
If possible suppose that, $\mathrm{Ap}(P_{2^{r}+1}(n),s_{0}) = \{a_{1}s_{1}+\dots+a_{n+r}s_{n+r} \mid (a_{1},\dots,a_{n+r})\in P(r,n)\}$. Then $ |\mathrm{Ap}(P_{2^{r}+1}(n),s_{0})| = |\{a_{1}s_{1}+\dots+a_{n+r}s_{n+r} \,\mid\, (a_{1},\dots,a_{n+r})\in P(r,n)\}|=2^{n+r+1}-1 \neq s_{0}$.   
\end{remark}

Thus, it remains to find the elements of the set $\{a_{1}s_{1}+\dots+a_{n+r}s_{n+r} \,\mid\, (a_{1},\dots,a_{n+r})\in P(r,n)\}$ which belongs to $\mathrm{Ap}(P_{2^{r}+1}(n),s_{0})$. To do so, we first define the following sets:


$F_{1}=\big\{a_{1}s_{1}+ \dots +a_{n+r-1}s_{n+r-1}+s_{n+r} \,\mid\,  a_{i}\in \{0,1,2\} \text{ for } 1 \leq i \leq n+r-2,\, a_{n+r-1} \in \{1,2\} \text{ and if }\, a_{j}=2 \,\text{ for some }\,  j \text{ then } a_{i}=0 \text{ for } i <j\big\}$; and

$F_{2}= \Big( \bigcup\limits_{l=0}^{r-2} E_{l} \cup \{2s_{n+r}\} \Big)  \mathbin{\big\backslash} \{s_{1}+s_{n}+s_{n+r}, \,2s_{1}+s_{n}+s_{n+r},\,s_{n}+s_{n+r}\}$,
where $E_{l}=\big\{a_{1}s_{1} + \dots + a_{n+l}s_{n+l} + s_{n+r}\,\mid\, a_{i}\in \{0,1,2\} \text{ for } 1 \leq i \leq n+l-1, \,a_{n+l}\in \{1,2\} \text{ and if } a_{j}=2 \text{ then } a_{i}=0 \text{ for } i < j \big \}$. Take $F=F_{1} \cup F_{2}$.


\begin{restatable}{lemma}{NotApery}\label{Not Apery} Under the standing hypothesis and notation, the following equalities hold.
\begin{enumerate}
    \item[(a)] $s_{n+l}+s_{n+r}-s_{0}=((2^{n+r}-2^{n+l})+2^{n+1}+4)s_{0}+(2^{n+l}-2^{n}-3)s_{1}$, for $1 \leq l \leq r$;
    \item[(b)] $s_{i}+s_{n}+s_{n+r}-s_{0}=((2^{r}+1)2^{n}+2-(2^{i}-4))s_{0}+(2^{i}-4)s_{1}$ for $2 \leq i \leq n$;
    \item[(c)] $s_{1}+s_{i}+s_{n}+s_{n+r}-s_{0}=((2^{r}+1)2^{n}+2-(2^{i}-4))s_{0}+(2^{i}-3)s_{1}$ for $2 \leq i \leq n$;
\end{enumerate}
\end{restatable}
\begin{proof}
\noindent (a) Let $1 \leq l \leq r$. Consider 
\begin{align*}
    &(2^{n+r}-2^{n+l}+2^{n+1}+4)s_{0}+(2^{n+l}-2^{n}-3)s_{1}\\
    =&(2^{n+r}-2^{n+l}+2^{n+1}+4)((2^r+1)2^{n}+1)+(2^{n+l}-2^{n}-3)((2^r+1)2^{n+1}+1)\\
    =&(2^{r}+1)2^{n}(2^{n+r}-2^{n+l}+2\cdot2^{n}+4+2(2^{n+l}-2^{n}-3))+2^{n+r}+2^{n}+1\\
    =&(2^{r}+1)2^{n}(2^{n+r}+2^{n+l}-2)+2^{n+r}+2^{n}+1\\
    =&(2^{r}+1)2^{n}(2^{n+r}+2^{n+l}-1)+1\\
    =&(2^{r}+1)(2^{n+n+r})+1+(2^{r}+1)2^{n+n+l}+1-(2^{r}+1)2^{n}-1\\
    =&s_{n+r}+s_{n+l}-s_{0}.
\end{align*}
(b) Let $2 \leq i \leq n$. Consider
\begin{align*}
&((2^{r}+1)2^{n}+2-(2^{i}-4))s_{0}+(2^{i}-4)s_{1}\\
    =&((2^{r}+1)2^{n}+2-(2^{i}-4))((2^{r}+1)2^{n}+1)+(2^{i}-4)((2^{r}+1)2^{n+1}+1)\\
    =&(2^{r}+1)2^{n}((2^{r}+1)2^{n}+2-2^{i}+4+2\cdot 2^{i}-8)+(2^{r}+1)2^{n}+2\\
    =&(2^{r}+1)2^{n}((2^{r}+1)2^{n}+2^{i}-2)+(2^{r}+1)2^{n}+2\\
    =&(2^{r}+1)2^{n}((2^{r}+1)2^{n}+2^{i}-1)+2\\
    =&(2^{r}+1)2^{n+n+r}+1+(2^{r}+1)2^{n+n}+1+(2^{r}+1)2^{n+i}+1-((2^{r}+1)2^{n}+1)\\
    =&s_{n+r}+s_{n}+s_{i}-s_{0}.
\end{align*}
(c) Follows from the proof of part (b). \qed
\end{proof}

\noindent The following lemmas give the explicit description of the elements in the Ap\'{e}ry set $\mathrm{Ap}(P_{2^{r}+1}(n),s_{0})$.

\begin{lemma}\label{F}
Let $n>2$ be an integer. Then $ F\,\cap \, \mathrm{Ap}(P_{2^{r}+1}(n),s_{0})=\phi $. 
\end{lemma}
\begin{proof}
Let $ a_{1}s_{1}+\dots+ a_{n+r-1}s_{n+r-1}+s_{n+r}\in F_{1}$.  From Lemma \ref{Not Apery}(a), we have $s_{n+r-1}+s_{n+r}-s_{0} \in P_{2^{r}+1}(n)$. Since $a_{n+r-1}\in \{1,2\}$, we have $a_{1}s_{1}+\dots+ a_{n+r-1}s_{n+r-1}+s_{n+r}-s_{0} =a_{1}s_{1}+\dots+ (a_{n+r-1}-1)s_{n+r-1}+ s_{n+r-1} +s_{n+r}-s_{0} \in P_{2^{r}+1}(n)$.

Let $a_{1}s_{1}+\dots +a_{n+l}s_{n+l}+s_{n+r}\in F_{2}$ for $1 \leq l \leq r-2$. From Lemma \ref{Not Apery}(a), we have $s_{n+l}+s_{n+r}-s_{0} \in P_{2^{r}+1}(n)$. Similar argument as above implies that $a_{1}s_{1}+\dots+a_{n+l}s_{n+l}+s_{n+r}-s_{0}\in P_{2^{r}+1}(n)$. 

Let $a_{1}s_{1}+\dots +a_{n}s_{n}+s_{n+r}\in F_{2} \,(\text{i.e. } l=0)$. Note that $a_{i}\neq 0$ for some $ i \in \{2,\dots, n-1\}$. From Lemma \ref{Not Apery}(b) and (c), we have $s_{i}+s_{n}+s_{n+r}-s_{0} \in P_{2^{r}+1}(n)$ and $s_{1}+s_{i}+s_{n}+s_{n+r}-s_{0}\in P_{2^{r}+1}(n)$. Since $a_{i}\neq 0$ for $2\leq i \leq n-1$, we have  $a_{1}s_{1}+\dots +a_{n}s_{n}+s_{n+r}-s_{0} \in P_{2^{r}+1}(n)$.  

Finally, consider $2s_{n+r} \in F_2$. From Lemma \ref{Not Apery}(a), we have $2s_{n+r}-s_{0} \in P_{2^{r}+1}(n)$. \\
Thus, for any element of $F$ say $x$, we have $x-s_{0} \in P_{2^{r}+1}(n)$ and hence $ F \, \cap \, \mathrm{Ap}(P_{2^{r}+1}(n),s_{0})=\phi $. \qed 
\end{proof} 

\begin{lemma}\label{card-F}
Under the standing hypothesis and notation, we have $|F|= 2^{n+r}-2^n-2.$
\end{lemma}
\begin{proof}
Consider the set $L_{11}=\big\{a_{1}s_{1}+ \dots+ a_{n+r-1}s_{n+r-1}+s_{n+r} \,\mid\, a_{i}\in \{0,1\} \text{ for } 1\leq i \leq n+r-2 \text{ and } a_{n+r-1}=1\big\}$. Clearly, $| L_{11} | =2^{n+r-2}$. Now we construct a new set $L_{12}$ as follows: Let $a_{1}s_{1}+\dots+a_{n+r-1}s_{n+r-1}+s_{n+r} \in L_{11}$. Take the least index $m \in \{1,2,...,n+r-1\}$ for which $a_{m}=1$, add an element $b_{1}s_{1}+\dots+b_{n+r-1}s_{n+r-1}+s_{n+r}$ in $L_{12}$ with $b_{m}=2$ and $b_{j}=a_{j}$ for all $j \neq m$. Clearly, $| L_{12} |=2^{n+r-2}$. Note that $F_{1}$ is the disjoint union of $L_{11}$ and $L_{12}$. Hence, $|F_{1} |=2^{n+r-1}$. 

Consider the set $L_{21}=\big\{a_{1}s_{1}+\dots+a_{n+l}s_{n+l}+s_{n+r} \,\mid\, a_{i} \in \{0,1\} \text{ for } 1\leq i \leq n+l-1 \text{ and } a_{n+l}=1 \big\}$. Clearly, $|L_{21}|=2^{n+l-1}$. Now we construct a new set $L_{22}$ as follows: Let $a_{1}s_{1}+\dots+a_{n+l}s_{n+l}+s_{n+r} \in L_{21}$. Take the least index $m$ for which $a_{m}=1$, add an element $b_{1}s_{1}+\dots+b_{n+l}s_{n+l}+s_{n+r}$ in $L_{22}$ with $b_{m}=2
$ and $b_{j}=a_{j}$ for all $j \neq m$. Clearly, $|L_{22}|=2^{n+l-1}$. Note that $E_{l}$ is the disjoint union of $L_{21}$ and  $L_{22}$. Hence, $|E_{l} | = 2^{n+l}$. Thus we get, $|F_{2}| = \sum\limits_{l=0}^{r-2} | E_{l} | +1-3= \sum\limits_{l=0}^{r-2} 2^{n+l}-2 = 2^{n+r-1}-2^{n}-2$. 
Therefore, $|F |=| F_{1} | + | F_{2} |= 2^{n+r-1}+2^{n+r-1}-2^{n}-2=2^{n+r}-2^n-2.$  \qed
\end{proof}

\begin{theorem}\label{Apery set}
Let $n > 2$ be an integer. Then $$\mathrm{Ap}(P_{2^{r}+1}(n),s_{0})= \{a_{1}s_{1}+\dots+a_{n+r}s_{n+r} \,\mid\, (a_{1},\dots,a_{n+r})\in P(r,n)\}\setminus F.$$
\end{theorem}
\begin{proof}
Let $P'(r,n)=\{a_{1}s_{1}+\dots+a_{n+r}s_{n+r} \,\mid\, (a_{1},\dots,a_{n+r})\in P(r,n)\}\setminus F$. Now from Lemma \ref{AP} and Lemma \ref{F}, it is clear that $\mathrm{Ap}(P_{2^{r}+1}(n),s_{0})\subseteq P'(r,n)$). Note that from Lemma \ref{card-P} and Lemma \ref{card-F}, we have 
\[|P'(r,n)| = 2^{n+r+1}-1-(2^{n+r}-2^{n}-2)=s_{0}= |\mathrm{Ap}(P_{2^{r}+1}(n),s_{0})|. \]

Thus, $\mathrm{Ap}(P_{2^{r}+1}(n),s_{0})= \{a_{1}s_{1}+\dots+a_{n+r}s_{n+r} \,\mid\, (a_{1},\dots,a_{n+r})\in P(r,n)\}\setminus F$. \qed 
\end{proof}

\section{The Frobenius Problem}
In this section, we give the formula for the Frobenius number of the Proth numerical semigroup $P_{2^r+1}(n)$ for all $r \geq 1$. We recall Lemma \ref{size-of-apery-set} from Section \ref{section-apery-set}. 

Let us begin with some preliminary lemmas.

\begin{lemma}\label{relation}
Let $s \in P_{2^r+1}(n)$ such that $s \not\equiv 0 (\mathrm{mod } \,s_{0})$, then $s+1 \in P_{2^r+1}(n)$. Moreover, $w(i+1) \leq w(i)+1$ for $1 \leq i \leq s_0-1$.
\end{lemma}
\begin{proof}
Since $s\in P_{2^r+1}(n)$, there exist $a_{0},\dots,a_{n+r} \in \mathbb{N}$ such that $s=a_{0}s_{o}+\dots+a_{n+r}s_{n+r}$. If $s \not\equiv 0 (\mathrm{mod }\, s_{0})$  then there exist $i \in \{1,\dots,n+r\}$ such that $a_{i} \neq 0$ and we get, $s+1=a_{0}s_{0}+\dots+(a_{i}-1)s_{i}+\dots+a_{n+r}s_{n+r}+s_{i}+1$.

Now, $s_{i}+1=k2^{n+i}+1+1=2k^{n+i-1}+2=2s_{i-1}$. Hence, 
$s+1=a_{0}s_{o}+\dots+(a_{i-1}+2)s_{i-1}+(a_{i}-1)s_{i}+\dots+a_{n+r}s_{n+r} \in P_{2^r+1}(n)$. 

Moreover, by definition, $w(i) \not\equiv 0 (\mathrm{mod }\, s_{0})$ for $1 \leq i \leq s_{0}-1$. Thus, $w(i)+1 \in P_{2^r+1}(n)$. Now, $w(i)+1 \equiv i+1 (\mathrm{mod}\,s_{0})$. As $w(i+1)$ is the least element of $P_{2^r+1}(n)$ which is congruent with $i+1$ modulo $s_{0}$, we get $w(i+1) \leq w(i)+1$.   \qed
\end{proof}

\begin{lemma}\label{First-element}
Let $n > 2$ be an integer. Then 
\begin{enumerate}
    \item $w(2)=s_{1}+s_{n}+s_{n+r}$;
    \item $w(1)=2s_{1}+s_{n}+s_{n+r}$. Moreover, $w(1)-w(2)=s_{1}$.
\end{enumerate}
\end{lemma}
\begin{proof}
\noindent (1) Consider \begin{align*}
    s_{1}+s_{n}+s_{n+r}-2=& (2^{r}+1)2^{n+1}+1+(2^{r}+1)2^{n+n}+1+(2^{r}+1)2^{n+n+r}-1\\
    =&2\cdot(2^{r}+1)2^{n}+(2^{r}+1)2^{2n}(2^{r}+1)+1\\
    =&(2^{r}+1)\cdot2^{n}+1)^{2}=s_{0}^{2}.
\end{align*}
Therefore, $s_{1}+s_{n}+s_{n+r} \equiv  2(\mathrm{mod}s_{0})$.
From Lemma \ref{Apery set} we have, $s_{1}+s_{n}+s_{n+r} \in \mathrm{Ap}(P_{2^r+1}(n),s_{0})$. Thus, $w(2)=s_{1}+s_{n}+s_{n+r}$.

\noindent (2) Note that from (1) we have $s_{1}+s_{n}+s_{n+r}-2=s_{0}^{2}$. Now
\begin{align*}
    2s_{1}+s_{n}+s_{n+r}-1=&s_{1}+s_{n}+s_{n+r}+2(2^{r}+1)2^{n}+1-1\\
    =& s_{1}+s_{n}+s_{n+r}-2+2s_{0}=s_{0}^{2}+2s_{0}.
\end{align*}
Therefore, $ 2s_{1}+s_{n}+s_{n+r} \equiv 1(\mathrm{mod}s_{0}).$
Again From Lemma \ref{Apery set} we have, $2s_{1}+s_{n}+s_{n+r} \in \mathrm{Ap}(P_{2^r+1}(n),s_{0})$. Thus, $w(1)=2s_{1}+s_{n}+s_{n+r}$. Clearly, $w(1)-w(2)=s_{1}.$  \qed
\end{proof}

The next Lemma is due to Selmer \cite{selmer-linear} gives us the relation among the Frobenius number and Ap\'{e}ry Set.

\begin{lemma}(\cite{selmer-linear}, Proposition 5 in \cite{assi-numerical})\label{Formula-Fobenius}
Let $S$ be a numerical semigroup and let $n$ be a non-zero element of $S$. Then $\mathrm{F}(S)= \max (\mathrm{Ap}(S,n))-n$. 
\end{lemma}

\begin{lemma}\label{Max-Apery-set}
Under the standing notation, we have $$w(1)=\mathrm{max}(\mathrm{Ap}(P_{2^r+1}(n),s_{0})).$$
\end{lemma}
\begin{proof}
From Lemma \ref{relation}, $w(i+1) \leq w(i)+1$, for $1 \leq i \leq s_0-1$. Thus, we get $w(3) \leq w(2)+1 , w(4) \leq w(3)+1 \leq w(2)+2$. In general, for $3 \leq j \leq s_{0}-1$, we have $w(j)\leq w(2)+(j-2)$. Since $w(1)-w(2)=s_{1}$, we get $w(j) \leq w(1)-s_{1}+(j-2)=w(1)-(s_{1}-(j-2)) <w(1)$ as $s_{1}-(j-2)>0$. Therefore, $w(1)\geq w(i)$ for $0 \leq i \leq s_{0}-1$ and $w(1)=\mathrm{max}(\mathrm{Ap}(P_{2^r+1}(n),s_{0}))$. \qed
\end{proof}

Thus, from Lemma \ref{Formula-Fobenius} and \ref{Max-Apery-set} we obtain the following formula for the Frobenius number of $P_{2^r+1}(n)$.

\begin{theorem}\label{Frobenius}
Let $n>2$ be a positive integer. Then $\mathrm{F}(P_{2^r+1}(n))=2s_{1}+s_{n}+s_{n+r}-s_{0}$.
\end{theorem}

Next we define the genus of a numerical semigroup. 

\begin{definition}
\emph{Let $S$ be a numerical semigroup then the set $\mathbb{N}\setminus S$ is called \emph{set of gaps of $S$} and its cardinality is said to be \emph{genus of $S$} denoted by $g(S)$.}
\end{definition}

\begin{remark}\label{Remark-Genus}
It is well known that (see Lemma 3 in \cite{assi-numerical}), $g(S)\geq \frac{\mathrm{F}(S)+1}{2}$. 
\end{remark}

\begin{corollary}
Let $n>2$ be a positive integer. Then, $g(P_{2^r+1}(n)) \geq k(2^{n+1}+2^{2n-1}+2^{2n+r-1}-2^{n-1})+2$.
\end{corollary}

\section{Pseudo-Frobenius Numbers and Type}\label{section Pseudo-Frobenius Numbers}

Our purpose in this section is to give the pseudo-Frobenius set and the formula for the type of the Proth numerical semigroup $P_{2^r+1}(n)$ for all $r \geq 1$. Let us recall the definition of pseudo-Frobenius numbers.

Let $S$ be a numerical semigroup. An integer $x$ is a \emph{pseudo-Frobenius number} of $S$ if $x \in \mathbb{Z}\setminus S$ and $x+s \in S$ for all $s \in S \setminus \{0\}$. 

Consider the following relation on the set of integers $\mathbb{Z}$: $a \leq_{S} b $ if $b-a \in S$. Note that this relation is an order relation i.e., it is reflexive, transitive and antisymmetric (see \cite{rosales-numerical}). The next lemma characterizes pseudo-Frobenius numbers in terms of the Ap\'{e}ry set using the relation defined above. 

\begin{lemma}(Proposition 2.20 in \cite{rosales-numerical})\label{PF}
Let $S$ be a numerical semigroup and let $n$ be a nonzero element of $S$. Then $$\mathrm{PF}(S)=\{w-n \,\mid\, w \in maximals_{\leq S} (\mathrm{Ap}(S,n)\}.$$
\end{lemma}

\begin{remark}\cite{Thabit}\label{not in Apery}
If $w,w' \in \mathrm{Ap}(S,x)$, then $ w'-w \in S$ if and only if $w'-w \in \mathrm{Ap}(S,x)$. Hence $maximal_{\leq _{S}}(\mathrm{Ap}(S,x))=\big\{w \in \mathrm{Ap}(S,x) \,\mid\, w'-w \notin \mathrm{Ap}(S,x)\setminus \{0\} \text{ for all } w'\in \mathrm{Ap}(S,x)\big\}$.
\end{remark}

Let $n>2$ be an integer. We define the set $X$ as follows: $X= \{ (a_{1},\dots,a_{n+r}) \,\mid\, a_{1}s_{1}+\dots+a_{n+r}s_{n+r}\in F\}$. Let us consider $M(n)=P(r,n)\setminus X$. It is clear that maximal elements in $M(n)$ (with respect to the product order) are \\
$\bullet (2,1,\dots,1,1,0),\dots, (0,\dots,0,\overset{\substack{r\\ \downarrow}}{2},1,\dots,1,0), \dots,(0,\dots,0,2,0)$;\\
$\bullet (2,1,\dots,\overset{\substack{n-1\\ \downarrow}}{1},0,\dots,0,1),\dots,(0,\dots,0, 2,\overset{\substack{n-1\\ \downarrow}}{1},0,\dots,0,1)$;\\
$\bullet (0,\dots,0,\overset{\substack{n-1\\ \downarrow}}{2},0,\dots,0,1), (2,0,\dots,0,\overset{\substack{n\\ \downarrow}}{1},0,\dots,0,1)$.

As a consequence of Theorem \ref{Apery set}, we get the following lemma.

\begin{lemma}\label{relation-maximal}
Under the standing notation, we have 
  
$maximal_{\leq P_{2^{r}+1}(n)}(\mathrm{Ap}(P_{2^{r}+1}(n),s_{0}))=maximal_{\leq P_{2^{r}+1}(n)}\big\{\{2s_{i}+s_{i+1}+\dots+s_{n+r-1} \,\mid\, 1 \leq i \leq n+r-1\}\cup \{2s_{j}+s_{j+1}+\dots+s_{n-1}+s_{n+r} \,\mid\, 1 \leq j \leq n-2\}\cup \{2s_{n-1}+s_{n+r},2s_{1}+s_{n}+s_{n+r}\} \big\}$.
\end{lemma}

We are now already to give the main result of this section.

\begin{restatable}{theorem}{PFMaximals}\label{PF-maximals}
Let $n>2$ be an integer and let $P_{2^{r}+1}(n)$ be the Proth numerical semigroup associated to $n$. Then $maximal_{\leq P_{2^{r}+1}(n)}(\mathrm{Ap}(P_{2^{r}+1}(n),s_{0}))=\{2s_{i}+s_{i+1}+\dots+s_{n+r-1} \,\mid\, 1 \leq i \leq r\}\cup \{2s_{j}+s_{j+1}+\dots+s_{n-1}+s_{n+r} \,\mid\, 1 \leq j \leq n-2\}\cup \{2s_{1}+s_{n}+s_{n+r}\} $.
\end{restatable}

\begin{proof}
Let $i \in \{r+1,...,n+r-1\} $, then
\begin{align*}
    &2s_{i}+s_{i+1}+\cdots+s_{n-1}+s_{n+r}-(2s_{r+i}+s_{r+i+1}+\cdots+s_{n}+s_{n+r-1})\\
    &=k2^{n+i}+k2^{n+i}(2^{r}-1)+r+k2^{2n+r}+2-(k2^{2n}(2^{r}-1)+r+k2^{n+r+i}+1)\\&
    =(k2^{2n}+1)=s_{n}.
\end{align*}
Also, $2s_{1}+s_{n}+s_{n+r}-(2s_{n-1}+s_{n+r})=2s_{1}+k2^{n}+1-2(k2^{n-1}+1)=s_{2}$. 

Hence, we get $2s_{r+i}+s_{r+i+1}+\cdots+s_{n}+s_{n+r-1}\leq_{ P_{2^{r}+1}(n)}2s_{i}+s_{i+1}+\cdots+s_{n-1}+s_{n+r}  $
for $i \in \{r+1,...,n+r-1\} $ and $2s_{n-1}+s_{n+r}\leq_{P_{2^{r}+1}(n)}2s_{1}+s_{n}+s_{n+r}$. From Lemma \ref{relation-maximal} we obtain that $maximal_{\leq P_{2^{r}+1}(n)}(\mathrm{Ap}(P_{2^{r}+1}(n),s_{0}))=maximal_{\leq P_{2^{r}+1}(n)}\big\{\{2s_{i}+s_{i+1}+\dots+s_{n+r-1} \,\mid\, 1 \leq i \leq r\}\cup \{2s_{j}+s_{j+1}+\dots+s_{n-1}+s_{n+r} \,\mid\, 1 \leq j \leq n-2\}\cup \{2s_{1}+s_{n}+s_{n+r}\} \big\}$. 

Consider a set $L_{1}=\{p_{i}=2s_{i}+s_{i+1}+\dots+s_{n+r-1} \,\mid\, 1 \leq i \leq r\} $ and $L_{2}=\{q_{j}=2s_{j}+s_{j+1}+\dots+s_{n-1}+s_{n+r} \,\mid\, 1 \leq j \leq n-2\}$. Take $L=L_{1} \cup L_{2} \cup \{2s_{1}+s_{n}+s_{n+1}\}$. We show that $L=maximal_{\leq P_{2^{r}+1}(n)}(\mathrm{Ap}(P_{2^{r}+1}(n),s_{0}))$.

Thus, to conclude the proof, it is enough to show that, for any $x,y \in L$, $x \not \leq _{P_{2^{r}+1}(n)} y$.

Let $p_{i},p_{i+1} \in L_{1}$, then
\begin{align*}
    p_{i+1}-p_{i}&=2s_{i+1}+s_{i+2}+\dots+s_{n+r-1}-(2s_{i}+s_{i+1}+\dots+s_{n+r-1})\\
    &=-2s_{i}+s_{i+1}=-1.
\end{align*}
Thus, the difference between any two element of $L_{1}$ is smaller than $r<s_{0}$. Which implies that $p_{i} \not \leq _{P_{2^{r}+1}(n)} p_{j}$ for any $1 \leq i, j \leq r$ and $i \neq j$. 

Similarly,  one can check that for $q_{i},q_{i+1} \in L_{2}$, $q_{i+1}-q_{i}=-1$ and  $q_{i} \not \leq _{P_{2^{r}+1}(n)} q_{j}$ for any $1 \leq i, j \leq n-2$ and $i \neq j$.

Let $p_{i}\in L_{1}$ and $q_{j} \in L_{2}$. Note that, $q_{1}-p_{1}=s_{n+r}-(s_{n}+\dots+s_{n+r-1})=k 2^{2n}+1-r$. 
Now consider $q_{j}-p_{i}=  q_{1}-(j-1)-(p_{1}-(i-1)) =  q_{1}-p_{1}-(j-i)= k 2^{2n}+1-r-j+i$. 

Suppose that $k 2^{2n}+1-r-j+i \in P_{2^{r}+1}(n)$, then there exists  $\lambda_0,\lambda_1,...,\lambda_{n+r} \in \mathbb{N}$ such that 
\begin{align*}
 k 2^{2n}+1-r-j+i&=   \lambda_0 s_0 + \lambda_1 s_1 + \cdots + \lambda_{n+r} s_{n+r}\\
 &=(\lambda_0 + \cdots +  \lambda_{n+r})+ k 2^{n} (\lambda_0 + 2 \lambda_1 + \cdots + 2^{n+r} \lambda_{n+r}).
\end{align*}
We get, $(\lambda_0 + \cdots + \lambda_{n+r})=1-r-j+i \leq 0$ which is a contradiction as $\lambda_i \in \mathbb{N}$. Thus, $q_{j}-p_{i} \notin P_{2^{r}+1}(n)$ and hence  $p_{i} \not \leq _{P_{2^{r}+1}(n)} q_{j}$  for $1 \leq i \leq r$, $1 \leq j \leq n-2$.

\noindent Now consider,
\begin{align*}
    2s_{1}+s_{n}+s_{n+r}-p_{i}&=2s_{1}+s_{n}+s_{n+r}-(p_{1}-(i-1))\\
    &=-s_{2}-\cdots-s_{n-1}-s_{n+1}-\cdots-s_{n+r-1}+s_{n+r}+i-1\\
    &=(k2^{n+2}-(n-3))+k2^{2n}-r+1+(i-1)\\
    &=k2^{n}(4+2^n)-n-r+3+i.
\end{align*}
If possible suppose that $k2^{2n}+k2^{n+2}-n-r+3+i \in P_{2^{r}+1}(n)$, then there exists  
$\lambda_0,\lambda_1,...,\lambda_{n+r} \in \mathbb{N}$ such that 
\begin{align*}
     k2^{n}(4+2^n)-n-r+3+i&= \lambda_0 s_0 + \lambda_1 s_1 + \cdots + \lambda_{n+r} s_{n+r}\\
     &=(\lambda_0  + \cdots +  \lambda_{n+r})+ k 2^{n} (2^{0} \lambda_0 + \cdots + 2^{n+r} \lambda_{n+r}).
\end{align*}
We get, $(\lambda_0 + \cdots +  \lambda_{n+r})=-(n+r-3-i)\leq 0$, which is a contradiction as $\lambda_i \in \mathbb{N}$. Therefore, $p_{i} \not \leq _{P_{2^{r}+1}(n)} 2s_{1}+s_{n}+s_{n+r}$  for $1 \leq i \leq r$.

Similarly, it is clear that $ 2s_{1}+s_{n}+s_{n+r}-q_{j}=k2^{n+2}+(j-n+2) \notin P_{2^{r}+1}(n)$. Therefore, $q_{j} \not \leq _{P_{2^{r}+1}(n)} 2s_{1}+s_{n}+s_{n+r}$  for $1 \leq j \leq n-2$.

Hence, difference between any two elements of $L$ do not belongs to $P_{2^{r}+1}(n)$. Thus, from Remark \ref{not in Apery}, we have $L=maximal_{\leq P_{2^{r}+1}(n)}(\mathrm{Ap}(P_{2^{r}+1}(n),s_{0}))$.  \qed

\end{proof}






By applying Lemma \ref{PF} and Theorem \ref{PF-maximals} we obtained the following theorem.

\begin{theorem}\label{type}
Let $n>2$ be an integer and let $P_{2^{r}+1}(n)$ be the Proth numerical semigroup. Then 

$  \mathrm{PF}(P_{2^{r}+1}(n))=\{2s_{i}+s_{i+1}+\dots+s_{n+r-1}-s_{0} \,\mid\, 1 \leq i \leq r\}\,\cup $
$\{2s_{j}+s_{j+1}+\dots+s_{n-1}+s_{n+r}-s_{0} \,\mid\, 1 \leq j \leq n-2\}$
$\cup \,\{2s_{1}+s_{n}+s_{n+r}-s_{0}\}$ \\
and $ \mathrm{t}(P_{2^{r}+1}(n))=| \mathrm{PF}(P_{2^{r}+1}(n) | =r+n-1$.
\end{theorem}

\section{Wilf's Conjecture}

In this section, we prove that the Proth numerical semigroup $P_{2^r+1}(n)$ supports Wilf's conjecture. Let us begin with the statement of Wilf's conjecture.

\begin{conjecture}\cite{Wilf-conjecture}
Let $S$ be a numerical semigroup, and $\mathrm{\nu}(S)=|\{s \in S \,\mid\, s \leq \mathrm{F}(S)\}|$, then $$\mathrm{F}(S)+1 \leq \mathrm{e}(S)\mathrm{\nu}(S),$$ where $\mathrm{e}(S)$ is the embedding dimension of $S$ and $\mathrm{F}(S)$ is the Frobenius number of $S$.
\end{conjecture}

\begin{lemma}\label{CONJ}(Corollary 5 in \cite{assi-numerical})
Let $S$ be a numerical semigroup. We have $ \mathrm{F}(S)+1 \leq  (\mathrm{t}(S)+1) \mathrm{\nu}(S)$.
\end{lemma}

From the previous lemma we obtain the following theorem.

\begin{theorem}
The Proth numerical semigroup $P_{2^{r}+1}(n)$ satisfies Wilf's conjecture. 
\end{theorem}

\begin{proof}
Recall that $\mathrm{e}(P_{2^{r}+1}(n))=n+r+1$ and from Lemma \ref{CONJ} 
\begin{align*}
    \mathrm{F}(P_{2^{r}+1}(n))+1 \leq &\,(\mathrm{t}(P_{2^{r}+1}(n))+1)\,\mathrm{\nu}(P_{2^{r}+1}(n))\\
    = &\, (n+r)\,\mathrm{\nu}(P_{2^{r}+1}(n))\\
    < &\, (n+r+1)\,\mathrm{\nu}(P_{2^{r}+1}(n))\\
    = &\,\mathrm{e}(P_{2^{r}+1}(n))\,\mathrm{\nu}(P_{2^{r}+1}(n). 
\end{align*} \qed
\end{proof}
\allowdisplaybreaks

\section{Conclusion} 

In this work, we obtained the formula for the embedding dimension of the Proth numerical semigroup $P_{k}(n)$. As a main result, we solved the Frobenius problem for $P_{2^{r}+1}(n)$. Moreover, we also attained the pseudo-Frobenius set and the type of $P_{2^{r}+1}(n)$. We concluded the paper by examining that $P_{2^{r}+1}(n)$ supports Wilf's conjecture. The following is an immediate open question to investigate: Is there a formula to find the Frobenius number and other invariants of the Proth numerical semigroup $P_{k}(n)$ for arbitrary $k$?

\end{document}